\documentclass[12pt]{amsart}
\usepackage[left=2.5cm,right=2.5cm, top=2.5cm, bottom=2.5cm]{geometry}
\usepackage{amsmath,amsfonts,amssymb}
\usepackage{color,soul} 
\usepackage{enumerate,graphics, graphicx}
\usepackage[pdftex,bookmarks=true]{hyperref}

\newtheorem{theorem}{Theorem}[section]
\newtheorem{lemma}[theorem]{Lemma}

\newtheorem{proposition}[theorem]{Proposition}
\newtheorem{prop}[theorem]{Proposition}
\newtheorem{definition}[theorem]{Definition}
\newtheorem{corollary}[theorem]{Corollary}
\theoremstyle{definition}
\newtheorem{defn}[theorem]{Definition}
\newtheorem{eg}[theorem]{Example}

\newcommand{\N}{{\mathbb N}}

\newcommand{\PP}{{\mathbb P}}

\newcommand{\W}{\mathcal{W}}
\newcommand{\E}{\mathcal{E}}

\newcommand{\Gonered}{\varGamma_{1_{red}}}
\newcommand{\Goneblue}{\varGamma_{1_{blue}}}
\newcommand{\Gtwored}{\varGamma_{2_{red}}}
\newcommand{\Gtwoblue}{\varGamma_{2_{blue}}}

\DeclareMathOperator{\Tan}{Tan}
\DeclareMathOperator{\maxdeg}{maxdeg}

\DeclareMathOperator{\supp}{supp}

\begin{document}
\title[Combinatorial degree bound for toric ideals of hypergraphs]{Combinatorial degree bound for toric ideals of hypergraphs}
\author[Gross and Petrovi\'c]{Elizabeth Gross and Sonja Petrovi\'c}

\begin{abstract}
Associated to any hypergraph is a toric ideal encoding the algebraic relations among its edges. We study these ideals and the combinatorics of their minimal generators, and derive  general degree bounds for   both uniform and non-uniform hypergraphs in terms of balanced hypergraph bicolorings, separators, and splitting sets. In turn, this provides  complexity bounds for algebraic statistical models associated to hypergraphs. As two main applications, we recover a  well-known complexity result for Markov bases of arbitrary $3$-way tables, and we  show that the defining ideal of  the tangential variety is generated by quadratics and cubics in cumulant coordinates. 
\end{abstract}

\maketitle


\section{Introduction}
\label{sec:intro}

The edge subring of a hypergraph $H$ is the monomial subalgebra parametrized by the edges of $H$.  We derive a general degree bound for the minimal generators of its defining ideal, $I_H$, in terms of the structure of the underlying hypergraph. 

Let $H$ be a hypergraph on  $V=\{1,\dots, n\}$  with edge set $E$.  Each edge $e_i\in E$ of size $d$ encodes a squarefree monomial
$
	x^{e_i}:=\prod_{j\in e_i} x_j
$
of degree $d$ in the polynomial ring $k[x_1,\dots,x_n]$.
The \emph{edge subring} of the hypergraph $H$, denoted by $k[H]$, is the following monomial subring:
\[
	k[H] := k[x^{e_i} : e_i \in E(H)].
\]
The \emph{toric ideal of $k[H]$}, denoted $I_H$,  is the kernel of the monomial map $\phi_H : k[t_{e_i}]  \to k[H]$ defined by 
$ \phi_H(t_{e_i})=x^{e_i}$.
The ideal $I_H$  encodes the algebraic relations among the edges of the hypergraph. For the special case where $H$ is a graph, generating sets of the toric ideal of $k[H]$ have been studied combinatorially in \cite{OH, OH2005}, \cite{RTT}, \cite{TT}, and \cite{Vill95, Vill}.

The motivation for studying toric ideals $I_H$ is threefold. First,  explicit results that relate $I_H$ to combinatorial properties of $H$  have existed only for graphs. Second, $I_H$ is related to the Rees algebra $\mathcal R(\mathcal I(H))$ of the edge ideal of $H$: in case when $H$ is a graph, the presentation ideal of $\mathcal R(\mathcal I(H))$  is completely determined by generators of $I_H$ and the syzygies of the edge ideal $\mathcal I(H)$. Third, these toric ideals correspond to Markov bases for algebraic statistical models; this connection is outlined at the end of the Introduction. 
A starting point of our work is the fact that combinatorial signatures of generators of $I_H$ are balanced edge sets of $H$. Balanced edge sets on  uniform hypergraphs were introduced in \cite{PS}, and are referred to as  monomial walks.

This paper is based on the fact that the ideal $I_H$ is generated by binomials $f_{\mathcal E}$ arising from primitive balanced edge sets $\mathcal E$ of ${H}$ (See Proposition~\ref{prop:balanced-non-uniform}, a generalization of \cite[Theorem 2.8]{PS}). A \emph{balanced edge set of $H$} is a multiset of bicolored edges $\mathcal E=\mathcal E_{blue} \sqcup \mathcal E_{red}$  satisfying the following balancing condition: 
for each vertex $v$ covered by $\mathcal E$, the number of red edges containing $v$ equals the number of blue edges containing $v$, that is, 
\[	\label{degreeCondition}\tag{*}
	\deg_{blue} (v) = \deg_{red} (v).
\]
A binomial $f_\mathcal E$  arises from $\mathcal E$ if it can be written as 
\[ 
f_\mathcal E=\prod_{e\in \mathcal E_{blue}} t_e - \prod_{e'\in\mathcal E_{red}} t_{e'}.
\]

 Note that while $H$ is a simple hypergraph (it contains no multiple edges), $\mathcal E$ allows repetition of edges. 
In addition, the balanced edge set $\E$ is \emph{primitive} if there exists no other balanced edge set $\E'=\E'_{blue}\sqcup\E'_{red}$ such that $\E'_{blue}\subsetneq \E_{blue}$ and $\E'_{red}\subsetneq \E_{red}$; this is the usual definition of an element in the Graver basis of $I_H$. If $H$ is a uniform hypergraph, a balanced edge set is called a \emph{monomial walk} to conform with the terminology in \cite{Vill95, Vill} and \cite{PS}.

In what follows, we give two general degree bounds for generators of $I_H$ (Section~\ref{sec:uniform}), study the combinatorics of splitting sets and reducibility (defined in Section~\ref{sec:splits-and-reducibility}), and explore implications to algebraic statistics throughout. 
Section~\ref{sec:indispensable} focuses on indispensable binomials, i.e. binomials that are members of every minimal generating set of $I_H$.  Proposition \ref{prop:indispensable} gives a combinatorial sufficient condition for determining whether a binomial $f \in I_H$ is indispensable. Consequently, the Graver basis is the unique minimal generating set of $I_H$  for any $2$-regular hypergraph (Proposition \ref{prop:2regular}).  In particular, this means that the Graver basis is equal to the universal Gr\"{o}bner basis, although the defining matrix need not be unimodular. 

 Theorem~\ref{thm:uniformDegreeBound} is a combinatorial criterion for the ideal  of a uniform hypergraph to be generated in degree at most $d\geq 2$. 
The criterion is based on decomposable balanced edge sets, separators, and splitting sets; see Definitions~\ref{defn:walkReducible} and~\ref{defn:splittingSet}.  
Our result generalizes the well-known criterion for the toric ideal of a graph to be generated in degree $2$ from  \cite{OH} and \cite{Vill95, Vill}. 
Splitting sets translate and extend the constructions used in \cite{OH} and \cite{Vill95, Vill}  to hypergraphs and arbitrary degrees. 
 Theorem~\ref{thm:nonuniform} provides  a more general result for non-uniform hypergraphs.

In algebraic statistics, any log-linear statistical model corresponds to a toric variety whose defining ideal gives a Markov basis for the model (cf. Fundamental Theorem of Markov bases \cite{DiacSturm}, \cite{algStatBook}). 
Since these varieties, by definition, have a monomial parametrization, we can also associate to any log-linear model $\mathcal M$ with a square-free parameterization a (non-uniform) hypergraph $H_\mathcal M$.  
By Proposition~\ref{prop:balanced-non-uniform}, Markov moves for the model $\mathcal M$ are  described by balanced edge sets of $H_\mathcal M$: if $\E$ is a balanced edge set of $H_\mathcal M$, then a Markov move on a fiber of the model  corresponds to replacing the set of red edges in $\E$ by the set of blue edges in $\E$. 
Our degree bounds give a bound for the \emph{Markov complexity} (Markov width) of the model $\mathcal M$. 
For general references on Markov complexity of classes for some log-linear models, the reader should refer to  \cite{DLO}, \cite{DS}, \cite[Chapter 1, \S 2]{algStatBook} and \cite{HSull}. 

We apply our combinatorial criteria to recover a well-known result in algebraic statistics from \cite{DLO} in Corollary~\ref{cor:SlimTables}. Finally,  
we study the Markov complexity of a set of models from \cite{SZ} called hidden subset models. Namely, Theorem~\ref{thm:cumulants} says that the ideal associated to the image of $\Tan((\PP^1)^n)$ in higher cumulants is generated by quadratics and cubics.


\section{Preliminaries and notation}

We remind the reader that all hypergraphs in this paper are simple, that is, they contain no multiple edges. In contrast, balanced edge sets of hypergraphs are not, since the binomials arising from the sets need not be squarefree. 
Therefore, for the purpose of this manuscript, we will refer to a balanced edge set  as a \emph{multiset} of edges, with implied vertex set; and, as usual, $V(\E)$ denotes the vertex set contained in the edges in $\E$.

For the remainder of this short section, we will clear  the technical details and notation we need for the proofs that follow. 

A multiset, $M$, is an ordered pair $(A, f)$ such that $A$ is a set and $f$ is a function from $A$ to $\N_{>0}$ that records the multiplicity of each of the elements of $A$.  
For example,  the multiset $M=(\{1,2\}, f)$ with $f(1)=1$ and $f(2)=3$ represents $M=\{1, 2, 2, 2\}$ where ordering doesn't matter.  We will commonly use the latter notation.

Given a multiset $M=(A, f)$,  \emph{the support of $M$} is $\supp (M):=A$, and its size is $|M|:= \sum_{a \in A} f(a)$. For two multisets $M_1=(A, f_1)$ and $M_2=(B, f_2)$, we say $M_2 \subseteq M_1$ if $B \subseteq A$ and for all $b \in B$, $f_2(b) \leq f_1(b)$.   
$M_2$ is a proper submultiset of $M_1$ if $B \subsetneq A$, or there exists a $b \in B$  such that $f_2(b)<f_1(b)$.

Unions, intersections, and relative complements of multisets are defined in the canonical way:

\begin{align*}
M_1 \cup M_2 &:=(A \cup B, g) \text{ where } g(a)= \begin{cases}
      f_1(a) & \text{ if } a \in A \setminus B, \\
      f_2(a) & \text{ if } a \in B \setminus A, \\
      \max(f_1(a), f_2(a)) & \text{ if } a \in A \cup B;
\end{cases}  \\	
\end{align*}
\begin{align*}
M_1 \cap M_2 &:=(A \cap B, g) \text{ where } g(a)= \min(f_1(a), f_2(a));  \\
\end{align*}
\begin{align*}
M_1 - M_2 &:=(C, g) \text{, where } g(a)=\begin{cases}
      f_1(a)& \text{ if } a \in A\setminus B, \\
      f_1(a)-f_2(a)& \text{otherwise}.
\end{cases}\\
&\text{and } C=A\setminus B \cup \{a \in A\cap B \ | \ f_1(a)-f_2(a)>0\}
\end{align*}

Note that the support of the union (intersection) of two multisets is the union (intersection) of their supports. 
Finally, we define a \emph{sum} of $M_1$ and $M_2$:

\[
	M_1 \sqcup M_2:=(A \cup B, g) \text{ where } g(a)=\begin{cases}
	f_1(a) & \text{ if } a\in A\setminus B, \\
	f_2(a) & \text{ if } a \in B \setminus A\\
	f_1(a)+f_2(a) \text{ if } a \in A \cap B
	\end{cases}.
\]

If $M_{1}\sqcup M_{2}$ is a balanced edge set, then the notation $M_{1}\sqcup_b M_{2}$ will be used to record the bicoloring of $M_1\sqcup M_2$: edges in $M_1$ are blue, and edges in $M_2$ are red. 

Finally, the number of edges in a hypergraph $H$ containing a vertex $v$ will be denoted by $\deg(v;H)$. For a bicolored multiset $M:=M_{blue}\sqcup_m M_{red}$, the blue degree $\deg_{blue} (v;M)$ of a vertex $v$ is defined to be $\deg(v;M_{blue})$.  The red degree  $\deg_{red} (v; M)$ is defined similarly.

\section{Splitting sets and reducible edge sets}
\label{sec:splits-and-reducibility}

The aim of this section is to lay the combinatorial groundwork for studying toric ideals of hypergraphs. In particular, we explicitly state what it combinatorially means for a binomial arising from a monomial walk to be generated by binomials of a smaller degree. 
We begin by describing the binomial generators of $I_H$.  Unless otherwise stated, $H$ need not be uniform.  

\begin{prop}\label{prop:balanced-non-uniform}
	Every binomial in the toric ideal of a hypergraph corresponds to a balanced edge set.  In particular, the toric ideal $I_H$ is generated by  primitive balanced edge sets. 
\end{prop}
\begin{proof}
	Suppose $\mathcal E$ is a balanced multiset of edges over $H$. Define  a binomial $f_\mathcal E\in k[t_{e} : e\in E(H)]$ as follows: 
	\[
		 f_\mathcal E=\prod_{e\in \mathcal E_{blue}} t_e - \prod_{e'\in\mathcal E_{red}} t_{e'}.
	\]
	The balancing condition~\eqref{degreeCondition} ensures that $f_\mathcal E$ is in the kernel of the map $\phi_H$.
	
	The second claim is immediate. 
\end{proof}

 Motivated by the application of reducible simplicial complexes to understand the Markov bases of hierarchical log-linear models \cite{DobraSul}, we now introduce notions of reducibility and separators for balanced edge sets. For simplicity,  we will often abuse notation and use $H$ to denote the edge set of $H$.

\begin{defn} 
\label{defn:walkReducible}

A balanced edge set $\E$ is said to be \emph{reducible with separator $S$}, $\supp (S)  \subseteq \supp(\E)$, and \emph{decomposition} $(\varGamma_1, S, \varGamma_2)$, if there exist balanced edge sets $\varGamma_1 \neq \E $ and $\varGamma_2 \neq E$  with $S\neq\emptyset$ 
such that  $S=\varGamma_{1_{red}} \cap \varGamma_{2_{blue}} $, $\E=\varGamma_1 \sqcup \varGamma_2$, and the  following coloring conditions hold:  $\Gonered, \Gtwored \subseteq \E_{red}$ and $\Goneblue, \Gtwoblue  \subseteq \E _{blue}$. 

We say that $S$ is \emph{proper}  with respect to $(\varGamma_1, S, \varGamma_2)$ if $S$ is a proper submultiset of both $\Gonered$ and $\Gtwoblue$.  

If $S$ is not proper, then $S$ is said to be \emph{blue} with respect to $(\varGamma_1, S, \varGamma_2)$ if $\Gonered=S$, and \emph{red} with respect to $(\varGamma_1, S, \varGamma_2)$ if $\Gtwoblue=S$. 
\end{defn}

Figure~\ref{fig:figure1} shows an example of a reducible balanced edge set $\E$. The separator is proper and consists of the single green edge $e_s$; it appears twice in the balanced edge set $\E$, once as a blue edge and once as a red edge. Figure~\ref{fig:figure2} shows a reducible balanced edge set where the separator, consisting of  the two green edges $e_1$ and $e_2$, is not proper. As before,  the separator edges appear twice in the balanced edge set.  

\begin{figure}[h]
\hfill
  \begin{minipage}[t]{.40\textwidth}
    \begin{center}  
\includegraphics[width=1.5in]{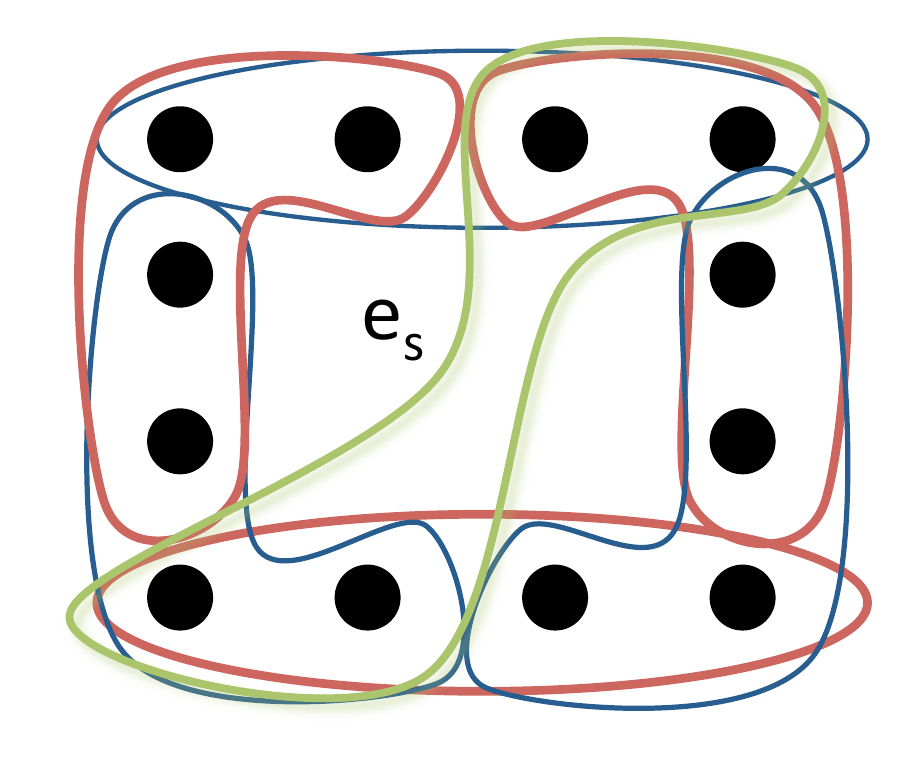}
\caption{Reducible balanced edge set. The green edge $e_s$ is the separator.}
\label{fig:figure1}
    \end{center}
  \end{minipage}
  \hfill
  \begin{minipage}[t]{.50\textwidth}
    \begin{center}  
\includegraphics[width=1.5in]{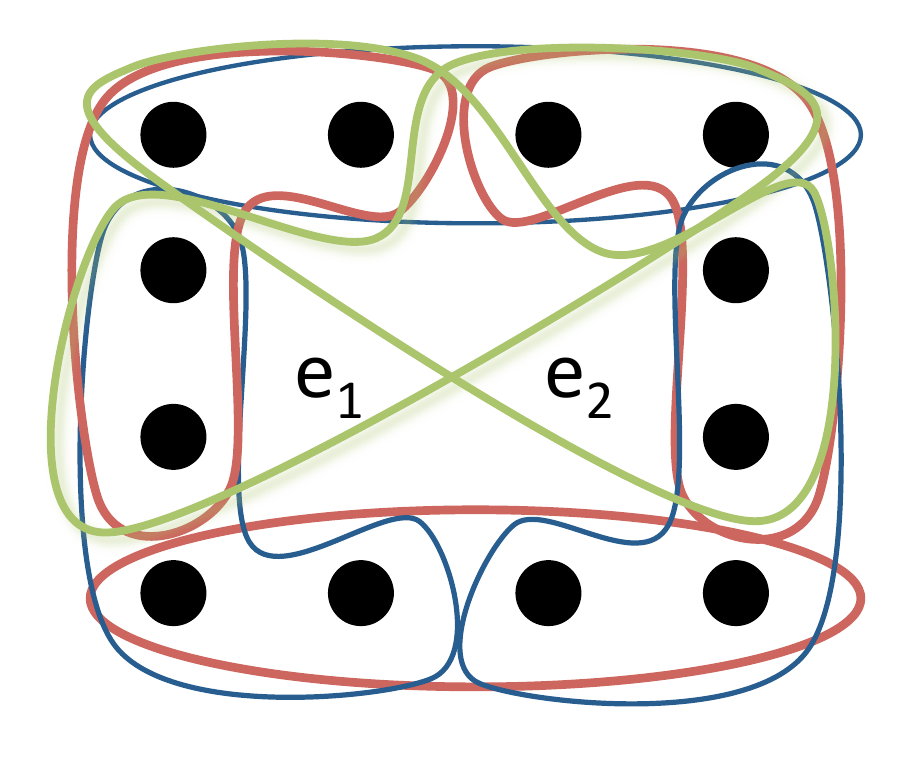}
\caption{Reducible balanced edge set with an improper separator. The separator consists of green edges $e_1$ and $e_2$.}
\label{fig:figure2}
    \end{center}
  \end{minipage}
  \hfill
\end{figure}

If $H$ is a hypergraph and $\E$ is a balanced edge set with $\supp (\E) \subseteq H$, given a multiset $S$ with $\supp(S) \subseteq H$, we can construct a new balanced edge set in the following manner:
\[
 	\E+S:=(\E_{blue} \sqcup S) \sqcup_m (\E_{red} \sqcup S).
\]

 \begin{defn} 
 \label{defn:splittingSet}
 Let $H$ be a hypergraph. Let $\E$ be a balanced edge set with size $2n$ such that $\supp(\E) \subseteq H$.  A non-empty multiset $S$ with $\supp (S) \subseteq H$ is a \emph{splitting set} of $\E$ with decomposition $(\varGamma_1, S, \varGamma_2)$ if $\E + S$ is reducible with separator $S$ and decomposition $(\varGamma_1, S, \varGamma_2)$.  
 
  $S$ is said to be a \emph{blue} (\emph{red}, resp.)  \emph{splitting set with respect} to $(\varGamma_1, S, \varGamma_2)$, if $S$ is a blue (red, resp.) separator of $\E + S$ with respect to $(\varGamma_1, S, \varGamma_2)$.  
  
  $S$ is a \emph{proper splitting set} of $\E$ if there exists a decomposition $(\varGamma_1, S, \varGamma_2)$ of $\E+ S$ such that $S$ is a proper separator with respect to $(\varGamma_1, S, \varGamma_2)$.
\end{defn}

\begin{eg}[{\bf Group-based Markov model}] 
Let $V_1=\{x_1, x_2, x_3, x_4\}$, $V_2=\{y_1, y_2, y_3, y_4\}$, and $V_3=\{z_1, z_2, z_3, z_4\}$.   Let $V$ be the disjoint union of $V_1$, $V_2$, and $V_3$. Let $H$ be the 3-uniform hypergraph with vertex set $V$ and edge set:

\[
\begin{array}{cccc}e_{111}=\{x_1,y_1, z_1\} & e_{122}=\{x_1, y_2,z_2\} & e_{133}=\{x_1,y_3,z_3\} & e_{144}=\{x_1,y_4,z_4\} \\e_{221}=\{x_2, y_2, z_1\} & e_{212}=\{x_2, y_1, z_2\} & e_{243}=\{x_2,y_4,z_3\} & e_{234}=\{x_2,y_3,z_4\} \\e_{331}=\{x_3, y_3, z_1\} & e_{342}=\{x_3, y_4, z_2\}  & e_{313}=\{x_3,y_1,z_3\} & e_{324}=\{x_3, y_2,z_4\} \\e_{441}=\{x_4,y_4,z_1\} & e_{432}=\{x_4, y_3, z_2\} & e_{423}=\{x_4,y_2,z_3\} & e_{414} =\{x_4,y_1,z_4\}\end{array}
\]

The hypergraph $H$ has applications in algebraic phylogenetics: it represents the parametrization of a particular group-based model from  \cite[Example 25]{phylo}.

Consider the monomial walk
\[
	\W=\{e_{324},e_{111},e_{243},e_{432}\} \sqcup_m \{e_{122},e_{313},e_{234},e_{441} \}.
\]
Let $S=\{e_{133}, e_{212}\}$. Then $S$ is a splitting set of $\W$ with decomposition $(\varGamma_1, S, \varGamma_2)$ where
\begin{align*}
\Gamma_1&=\{e_{111},e_{243},e_{432}\} \sqcup_m \{e_{133},e_{212}, e_{441} \}\\
\Gamma_2 &=\{e_{133},e_{212},e_{324}\} \sqcup_m \{e_{122},e_{313},e_{234} \}.
\end{align*}

The decomposition $(\Gamma_1,S,\Gamma_2)$ encodes binomials in $I_H$ that generate $f_{\W}$:
\[
	f_{\W}=t_{e_{324}}(t_{e_{111}}t_{e_{243}}t_{e_{432}} - t_{e_{133}}t_{e_{212}}t_{ e_{441}})+
	t_{e_{441}}(t_{e_{133}}t_{e_{212}}t_{e_{324}} - t_{e_{122}}t_{e_{313}}t_{e_{234}}).   
\]
\end{eg}

The previous example illustrates the algebraic interpretation of a splitting set.  Notice there is a correspondence between monomials in $k[t_{e_i}]$ and multisets of edges of $H$. We will write $E(t_{e_{i_1}}^{a_1}t_{e_{i_2}}^{a_2}\cdots t_{e_{i_l}}^{a_l})$ for the multiset $( \{e_{i_1}, \ldots, e_{i_l}\}, f)$ where
\begin{align*}
f: \{e_{i_1}, \ldots, e_{i_l}\} &\to \N \\
e_{i_j} \mapsto a_{j}. 
\end{align*}
Thus the support of $E(t_{e_{i_1}}^{a_1}t_{e_{i_2}}^{a_2}\cdots t_{e_{i_l}}^{a_l})$ corresponds to the  support of the monomial 
$t_{e_{i_1}}^{a_1}t_{e_{i_2}}^{a_2}\cdots t_{e_{i_l}}^{a_l}$.

If $f_{\E}=u-v \in I_{H}$ is the binomial arising from the balanced edge set $\E$, then a monomial $s$ corresponds to a splitting set $S$ if and only if there exist two binomials $u_1 -v_1, u_2-v_2 \in I_{H}$ such that $us=u_1u_2$, $vs=v_1v_2$ and $s=\gcd(v_2,u_1)$.  In this case, the decomposition of $\E + S$ is $(\Gamma_1, S, \Gamma_2)$ where $\Gamma_1=E(u_1) \sqcup_m E(v_1)$ and $\Gamma_2=E(u_2) \sqcup_m E(v_2)$.

For a balanced edge set, $\E$, the existence of a spitting set determines whether the binomial $f_{E} \in I_H$ can be written as the linear combination of two binomials $f_{\Gamma_1}$,$f_{\Gamma_2}$ $\in I_H$. While, in general, the existence of a splitting set does not imply $\deg(f_{\Gamma_1}), \deg(f_{\Gamma_2}) < \deg (f_{\E})$, if $H$ is uniform and the splitting set is proper, then the following lemma holds.

\begin{lemma} \label{lm:properSplittingGiveDecomposition}
Let $H$ be a uniform hypergraph and let $\W$ be a monomial walk with $\supp(\W) \subseteq H$ and $|\W|=2n$. If $S$ is a proper splitting set of $\W$, then there exists a decomposition $(\varGamma_1, S, \varGamma_2)$ of $\W+S$ such that $|\varGamma_1| <  |\W|$ and $|\varGamma_2|<|\W|$. \end{lemma}

\begin{proof}  
Let $S$  be a proper splitting set of $\W$.  By definition, 
 there exists a decomposition $(\varGamma_1, S, \varGamma_2)$ of $\W+S$, such that $S$ is a proper submultiset of $\Gonered$ and $\Gtwoblue$.

 Let $|\varGamma_1|=2n_1$ and $|\varGamma_2|=2n_2$.  Since $\W+S=\varGamma_1 \sqcup \varGamma_2$, it follows that $|\W +S|=|\varGamma_1|+|\varGamma_2|$.  Then, $2n+2|S|=2n_1+2n_2$, which implies $2n-2n_1=2n_2 -2|S|$.
But $S$ being a proper submultiset of $\Gtwoblue$ gives that $n_2 > |S|$, which, in turn, implies that $n>n_1$.  By a similar argument, $n>n_2$.  Thus $|\varGamma_1| <  |\W|$ and $|\varGamma_2|<|\W|$.
\end{proof}

\section{Indispensable Binomials}
\label{sec:indispensable}

A binomial $f$ in a toric ideal $I$ is \emph{indispensable} if  $f$ or $-f$ belongs to every binomial generating set of $I$. Indispensable binomials of toric ideals were introduced by Takemura et al, and are studied in \cite{AA}, \cite{ATY}, \cite{CKT}, \cite{OH2005}, \cite{RTT}.

\begin{prop}\label{prop:indispensable} Let $H$ be a hypergraph. Let $\E$ be a balanced edge set with $\supp(\E) \subseteq H$. Let $f_{\E}$ be the binomial arising from $\E$. If there does not exist a splitting set of $\E$, then $f_{\E}$ is an indispensable binomial of $I_H$. 
\end{prop}

\begin{proof}  Suppose $\E$ is not indispensable. Then there is a binomial generating set of $I_H$, $\mathcal{G}=\{f_1, \ldots, f_n\}$, such that $f_{\E} \notin \mathcal{G}$ and $-f_{\E} \notin \mathcal{G}$.  

Since $f_{\E} = f_\E^+-f_\E^- \in I_H$, there is a $f_i=f_i^+-f_i^- \in \mathcal{G}$ such that $f_i^+$ or $f_i^-$ divides $f_{\E}^{+}$. Without loss of generality, assume $f_i^{+} | f_{\E}^{+}$. Since $f_i$ is a binomial in $I_H$, $f_i$ arises from a monomial walk $\E_i$ on $H$.

Let $S=\E_{i_{red}}$. Let $\Gamma_1=\E_i$ and $\Gamma_2=\Gtwoblue \sqcup_m \Gtwored$ where 
\begin{align*}
\Gamma_{2_{blue}}&= ((\E_{blue} - \E_{i_{blue}}) \sqcup \E_{i_{red}})\\
\Gtwored&=\E_{red}.
\end{align*}
 Since  $f_i^{+} | f_\E^{+}$, the  multiset $\E_{i_{blue}} \subseteq \E_{blue}$, and thus $\Gamma_1 \sqcup \Gamma_2=\E+S$.  By construction, $\Gonered \cap \Gtwoblue=S$. Therefore $S$ is a splitting set of $\E$.
\end{proof}

If every Graver basis element of a binomial ideal $I_H$ is indispensable, then the Graver basis of $I_H$ is the unique minimal generating set of $I_H$. Propositions~\ref{prop:2regular} and~\ref{prop:partite2reg} describe two classes of hypergraphs where this is the case.  In particular, for these hypergraphs, the universal Gr\"{o}bner basis of $I_H$ is a minimal generating set.

\begin{prop}\label{prop:2regular} If $H$ is a 2-regular uniform hypergraph, then the Graver basis of $I_H$ is the unique minimal generating set of $I_H$.
\end{prop}

For the proof of Proposition~\ref{prop:2regular}, we make use of Proposition 3.2 in \cite{PS} which concerns balanced edge sets that are pairs of perfect matchings.

\begin{defn} A \emph{matching} on a hypergraph $H=(V,E)$ is a subset $M \subseteq E$ such that the elements of $M$ are pairwise disjoint.  A matching is called \emph{perfect} if $V(M)=V$.
\end{defn}

\begin{proof}[Proof of Proposition~\ref{prop:2regular}] 
Let $\mathcal{G}$ be the Graver basis of $I_H$ and let $f \in \mathcal{G}$. Since every element of $\mathcal{G}$ is binomial,  $f$ arises from a primitive monomial walk $\W$ with $\supp(\W) \subseteq H$.

Let $M_b=\supp(\W_{red})$ and $M_r=\supp(\W_{blue})$.  By primitivity of $\W$, the intersection $M_r \cap M_b = \emptyset$. Since $\W$ satisfies condition~\eqref{degreeCondition} and $H$ is 2-regular, if $e_1, e_2 \in M_b$ and $e_1 \cap e_2 \neq \emptyset$, then $e_1 \in M_r$ or $e_2 \in M_r$, which would contradict the primitivity of $\W$.  So $M_b$ and $M_r$ are two edge-disjoint perfect matchings on $V(\W)$.  By Proposition 3.2 in \cite{PS}, $\W$ contains no multiple edges, i.e. $\W = M_b \sqcup_m M_r$.  Furthermore, since $H$ is 2-regular, the edge set of the subhypergraph induced by $V(\W)$ is $M_b \cup M_r$

Suppose $S$ is a splitting set of $\W$ with decomposition $(\varGamma_1, S, \varGamma_2)$.  By the correspondence between primitive monomial walks and primitive binomials, there exists a primitive monomial walk $\varGamma$ such that $\varGamma_{blue} \subseteq \Goneblue$ and $\varGamma_{red} \subseteq \Gonered$ (if $\varGamma_1$ is primitive, then $\varGamma=\varGamma_1$). By Proposition 3.2 in\cite{PS}, $\varGamma$ must be a pair of perfect matchings on $V(\varGamma)$. This means $\varGamma$ is a proper balanced edge set of $\W$, a contradiction.  Therefore, by Proposition \ref{prop:indispensable}, $f_{\W}$ is indispensable. Since every element in the Graver basis of $I_H$ is indispensable, there is no generating set of $I_H$ strictly contained in the Graver basis, and the claim follows. 
\end{proof}

\begin{definition} A $k$-uniform hypergraph $H=(V,E)$ is \emph{$k$-partite} if there exists a partition of $V$ into $k$ disjoint subsets, $V_1, \ldots, V_k$, such that each edge in $E$ contains exactly one vertex from each $V_i$.
\end{definition}

\begin{lemma}\label{lem:partite} Let $H=(V, E)$ be a $k$-uniform $k$-partite hypergraph with $E=E_b \sqcup E_r$ and $E_b \cap E_r = \emptyset$. If there exists a $V_i$, $1 \leq i \leq k$, such that  $\deg(v;E_r)=\deg(v;E_b)=1$  for all $v \in V_i$, then a monomial walk $\W$ with support $E$ is primitive only if $\W$ contains no multiple edges.
\end{lemma}

\begin{proof}  Follows from the proof of necessity of Proposition 3.2 in \cite{PS}.
\end{proof}

\begin{prop}\label{prop:partite2reg} 
Let $H=(V, E)$ be a $k$-uniform $k$-partite hypergraph.  If there exists a $V_i$ such that $\deg(v; E)=2$ where for all $v \in V_i$, then  the Graver basis of $I_H$ is the unique minimal generating set of $I_H$.
\end{prop}

\begin{proof} The proof is similar to the proof of Proposition~\ref{prop:2regular}. Note that while  $H$ may not be $2$-regular, one of its parts, $V_i$, is `locally' $2$-regular, and thus restricts the structure of monomial walks on $H$. In particular, Lemma~\ref{lem:partite} ensures that $M_r$ and $M_b$, are edge-disjoint perfect matchings on $V(\W)|_{V_i}$, and the rest of the proof follows immediately. 
\end{proof}

\begin{eg}[{\bf No 3-way interaction}] \label{ex:no3way}
The toric ideal of the hypergraph $H$ in Figure \ref{fig:nothreeway} corresponds to the hierarchical log-linear model for no $3$-way interaction on $2 \times 2 \times 2$ contingency tables.  This statistical model is a common example in algebraic statistics \cite[Example 1.2.7]{algStatBook}. 
Since there is exactly one primitive monomial walk $\W$ on $H$ that travels through $8$ edges, $I_H=(f_\W)$. 
\begin{figure}
\begin{center}  
\includegraphics[width=2.5in]{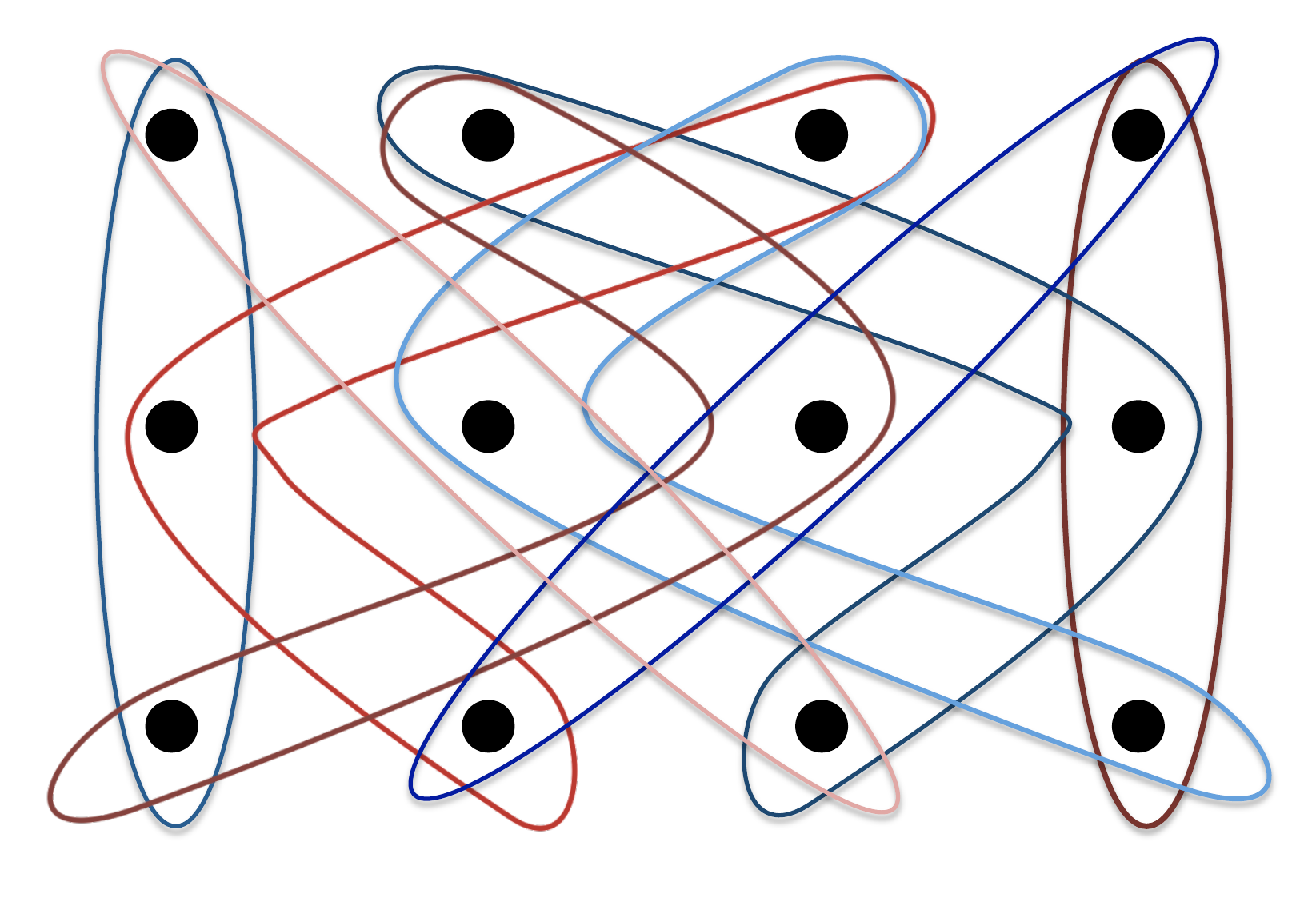}
\caption{}
\label{fig:nothreeway}
\end{center}
\end{figure}

For $2 \times 3 \times 3$ contingency tables with no 3-way interaction, the hypergraph  corresponding to this log-linear model has 18 edges.  The hypergraph in this case is $H=(V, E)$ where $V=\{x_{00}, x_{01}, x_{02}, x_{10}, x_{11},x_{12}, y_{00}, y_{01}, y_{02}, y_{10}, y_{11},y_{12}, z_{00}, z_{01}, z_{02}, z_{10}, z_{11},z_{12},z_{20},z_{21},z_{22}\}$ and the edge set is:
\[ \begin{array}{ccc}e_{000}=\{x_{00}, y_{00}, z_{00}\} & e_{001}=\{x_{00}, y_{01}, z_{01}\} & 
e_{002}=\{x_{00},y_{02},z_{02}\} \\ 
e_{010}=\{x_{01}, y_{00}, z_{10}\} & e_{011}=\{x_{01}, y_{01}, z_{11}\}  
& e_{012}=\{x_{01},y_{02},z_{12}\} \\ e_{020}=\{x_{02},y_{00},z_{20} \}&
e_{021}=\{x_{02},y_{01},z_{21} \}&
e_{022}=\{x_{02},y_{02},z_{22} \}\\
e_{100}=\{x_{10}, y_{10}, z_{00}\} & e_{101}=\{x_{10}, y_{11}, z_{01}\} & 
e_{102}=\{x_{10},y_{12},z_{02}\} \\ 
e_{110}=\{x_{11}, y_{10}, z_{10}\} & e_{111}=\{x_{11}, y_{11}, z_{11}\}  &
e_{112}=\{x_{11},y_{12},z_{12}\} \\
e_{120}=\{x_{12},y_{10},z_{20} \}&
e_{121}=\{x_{12},y_{11},z_{21} \}&
e_{122}=\{x_{12},y_{12},z_{22} \}
\end{array} \]

Let $\W$ be the primitive monomial walk
\[
\W=\{e_{000},e_{101}, e_{011}, e_{112}, e_{022}, e_{120} \} \sqcup_m \{e_{100},e_{001}, e_{111}, e_{012}, e_{122}, e_{220}. \}
\]

Every remaining edge $H$ that does not appear in $\W$ is not contained in $V(\W)$, thus it can be easily verified that there does not exist a splitting set of $\W$, so by Proposition \ref{prop:indispensable}, $f_{\W}$ is indispensable.  In fact, $H$ satisfies the condition of Proposition~\ref{prop:partite2reg} and thus every binomial in $I_H$ corresponding to a primitive monomial walk is indispensable.
\end{eg}

From the above discussion, we can see that if a uniform hypergraph $H$ contains an induced subhypergraph $H_s$ that is 2-regular and there exists a bicoloring such that with this bicoloring $H_s$ is also a balanced edge set, then the maximum degree of any minimal generating set of $I_H$ is at least $|E(H_s)|/2$. 

A similar statement holds for $k$-uniform, $k$-partite hypergraphs with vertex partition $V=\cup_{i=1}^k V_i$. Namely, if $H$  contains an induced subhypergraph $H_s$ that is 2-regular on $V_i$ (i.e., $H$ satisfies the conditions of Proposition~\ref{prop:partite2reg}) and there exists a bicoloring such that with this bicoloring $H_s$ is a balanced edge set (e.g., $H_s$ is a pair of disjoint perfect matchings), then the maximum degree of any minimal generating set of $I_H$ is at least $|E(H_s)|/2$.

Recall that degree bounds on minimal generators give a Markov complexity bound for the corresponding log-linear model in algebraic statistics. This allows us to recover a well-known result: 
\begin{corollary}[Consequence of Theorem 1.2 in \cite{DLO}; see also Theorem 1.2.17 in \cite{algStatBook}] 
\label{cor:SlimTables}
The Markov complexity for the no 3-way interaction model on $3 \times r \times c$ contingency tables grows arbitrarily large as $r$ and $c$ increase .
\end{corollary}
\begin{proof}
For the no 3-way interaction model on $2 \times r \times c$ contingency tables, we can construct a primitive binomial $f_{H_s}$ of degree $2 \cdot \min(r,c)$ in its defining toric ideal by taking a cycle of length $\min(r,c)$ on the bipartite graph $K_{r,c}$.  (We remind the reader that this is precisely how $f_{\W}$ is constructed in Example \ref{ex:no3way}).  By noting that the hypergraph associated to this binomial $H_s$ is an induced subhypergraph of the hypergraph associated to the $3 \times r \times c$ case and that $H_s$ is 2-regular in one of the partitions, the claim follows by Proposition~\ref{prop:partite2reg}.
\end{proof}

\section{General degree bounds and an application}
\label{sec:uniform}
\label{sec:nonuniform}
\label{sec:application}

For uniform hypergraphs, balanced edge sets are referred to as monomial walks. In the previous sections, we saw that splitting sets of $\W$ translate to algebraic operations on the binomials $f_{\W}$, providing a general construction for rewriting a high-degree binomial  in terms of binomials corresponding to shorter walks.  This, along with Lemma~\ref{lm:properSplittingGiveDecomposition}, is the key to the general degree bound result.

\begin{theorem}
\label{thm:uniformDegreeBound} 
	Given a $k$-uniform hypergraph $H$, the toric ideal  $I_H$ is generated in degree at most $d$ if and only if for every primitive monomial walk $\W$ of length $2n>2d$, with $\supp(\W) \subseteq H$, one of the following two conditions hold:

	i) there exists a proper splitting set $S$ of $\W$,

	\emph{or}

	ii) there is a finite sequence of pairs, $(S_1,R_1), \ldots, (S_N,R_N)$, such that

	\begin{itemize}
	\item $S_1$ and $R_1$ are blue and red splitting sets of $\W$ of size less than $n$ with decompositions $(\varGamma_{1_1}, S_1, 	\varGamma_{2_1})$ and $(\varUpsilon_{1_1}, R_1,\varUpsilon_{2_1})$,
	\item $S_{i+1}$ and $R_{i+1}$ are blue and red splitting sets of $\W_i=\varGamma_{2_{i_{blue}}} \sqcup_m \varUpsilon_{1_{i_{red}}}$ of size less than $n$ with decompositions $(\varGamma_{1_{i+1}}, S_{i+1}, \varGamma_{2_{i+1}})$ and $(\varUpsilon_{1_{i+1}}, R_{i+1,}\varUpsilon_{2_{i+1}})$, and,
	\item $S_N \cap R_N \neq \emptyset$ or there exists a proper splitting set of $\W_{N}$.
	\end{itemize}
\end{theorem}

\begin{proof}[Proof of necessity ($\Rightarrow$)]
Let $H$ be a $k$-uniform hypergraph whose toric ideal $I_H$ is generated in degree at most $d$. Let $\W$ be a primitive monomial walk of length $2n>2d$. Let $p_{\W}=u-v$ be the binomial that arises from ${\W}$. Since $I_H$ is generated in degree at most $d$, there exist primitive binomials of degree at most $d$, $(u_1-v_1), \ldots, (u_s-v_s) \in k[t_{e_i}]$, and  $m_1, \ldots, m_s \in k[t_{e_i}]$, such that
\[p_{\W}=m_1(u_1 -v_1)+m_2(u_2-v_2)+ \ldots + m_s(u_s-v_s).\]
By expanding and reordering so that $m_1u_1=u_w$, $m_sv_s=v_w$,  and $m_iv_i=m_{i+1}u_{i+1}$ for all $i=1, \ldots, s-1$, we may and will assume that $m_1,\dots, m_s$ are monomials. 

If $\gcd(m_i, m_{i+1}) \neq 1$ for some $i$, we can add the terms $m_i(u_i-v_i)$ and $m_{i+1}(u_{i+1}-v_{i+1})$ to get a new term, $m_i'(u_i'-v_i')$, where $m_i'=\gcd(m_i, m_{i+1})$ and $(u_i'-v_i')$ is an binomial of $I_H$ of degree less than $n$. Continuing recursively in the manner, we have
\[
	p_{\W}=m_1'(u_1' -v_1')+m_2'(u_2'-v_2')+ \ldots + m_r'(u_r'-v_r')
\]
where $m_1'u_1'=u_w'$, $m_r'v_r'=v_w'$, $m_i'v_i'=m_{i+1}'u_{i+1}'$, $\gcd(m_i', m_{i+1}')=1$ for all $i=1, \ldots, r-1$,  and $\deg (u_i'-v_i')<n$ for all $i=1,\ldots r$.  For convenience, we will drop the superscripts and write
\[p_w=m_1(u_1 -v_1)+m_2(u_2-v_2)+ \ldots + m_r(u_r-v_r).\]

{\bf Case 1:} $r=2$.
In this case,  $p_{\W}=m_1(u_1 -v_1)+m_2(u_2-v_2)$. Let
\begin{align*}
\varGamma_1&:=E(u_1) \sqcup_m E(v_1)\\
\varGamma_2&:=E(u_2) \sqcup_m E(v_2)\\
S&:=E(v_1) \cap E(u_2)=E(\gcd(v_1,u_2)).
\end{align*}
We want to show $(\varGamma_1, S, \varGamma_2)$ is a decomposition of $\W + S$. Since $S=\Gonered \cap \Gtwoblue$, $\Goneblue \subseteq \W_{blue}$, and $\Gtwored \subseteq \W_{red}$, we only need to show $\W+S=\varGamma_1 \sqcup \varGamma_2$, $\Gtwored \subseteq (\W + S)_{red}$, and $\Gtwoblue \subseteq (\W + S)_{blue}$.
First, notice the following equalities hold:
\begin{align*}
\W+S &= (\W_{blue} \sqcup S) \sqcup (\W_{red} \sqcup S)
	=E(u) \sqcup S \sqcup E(v) \sqcup S\\
	&= E(m_1u_1) \sqcup S \sqcup E(m_2v_2) \sqcup S
	 =E(m_1) \sqcup E(u_1) \sqcup S \sqcup E(m_2) \sqcup E(v_2) \sqcup S.
\end{align*}

Let $s\in k[t_{e_i}]$ be the monomial such that $E(s)=S$, so $s=\gcd(v_1,u_2)$.  The equality $m_1v_1=m_2u_2$ implies $m_1(\frac{v_1}{s})=m_2(\frac{u_2}{s})$. Now, $\frac{v_1}{s}$ and $\frac{u_2}{s}$ are clearly relatively prime, and by the assumptions on $p_{\W}$, $m_1$ and $m_2$ are relatively prime.  This means the equality $m_1(\frac{v_1}{s})=m_2(\frac{u_2}{s})$ implies $m_1=\frac{u_2}{s}$ and $m_2=\frac{v_1}{s}$.
Thus,
\begin{align*}
\varGamma_1 \sqcup \varGamma_2 &= 	E(u_1) \sqcup E(v_1) \sqcup E(u_2) \sqcup E(v_2)\\
&= E(u_1) \sqcup E(\frac{v_1}{s}) \sqcup S \sqcup E(v_2) \sqcup E(\frac{u_2}{s}) \sqcup S\\
&=E(u_1) \sqcup E(m_2) \sqcup S \sqcup E(v_2) \sqcup E(m_1) \sqcup S.
\end{align*}

Consequently, $\W+S=\varGamma_1 \sqcup \varGamma_2$.  

Notice the equality $m_2=\frac{v_1}{s}$ also implies $\Gonered=E(v_1)=E(m_2) \sqcup S$.  This means $\Gonered \subseteq (E(m_2u_2) \sqcup S)=(\W_{red} \sqcup S)=(\W+S)_{red}$.  By a similar observation, $\Gtwoblue \subseteq (\W+S)_{blue}$.

{\bf 
Case 2:} $r=2N+1$.
For $1< i <N$, let
\begin{align*}
\varGamma_{1_i}&=E(u_i) \sqcup_m E(v_i)\\
\varGamma_{2_i}&=E(m_{i+1}u_{i+1}) \sqcup_m E(m_{2N-i+2}v_{2N-i+2})\\
S_i&=E(v_i) \cap E(m_{i+1}u_{i+1})=E(\gcd(v_i,  m_{i+1}u_{i+1}))=E(v_i).
\end{align*}

For $1< i <N$, let 
\begin{align*}
\varUpsilon_{1_i}&=E(m_iu_i) \sqcup_m E(m_{2N-i+1}v_{2N-i+1})\\
\varUpsilon_{2_i}&=E(u_{2N-i+2}) \sqcup_m E(v_{2N-i+2})\\
R_i&=E(m_{2N-i+1}v_{2N-i+1}) \cap E(u_{2N-i+2})
\\&=E(\gcd(m_{2N-i+1}v_{2N-i+1},u_{2N-i+2}))=E(u_{2N-i+2}).
\end{align*}
One can follow the proof of Case 1) to see that $S_1$ and $R_1$ are splitting sets of $\W$, and $S_{i+1}$ and $R_{i+1}$ are splitting sets of $\W_i=E(m_{i+1}u_{i+1}) \sqcup_m E(m_{2N-i+1}v_{2N-i+1})$ for $i=1, \ldots, N-1$. Furthermore, by definition, they are blue and red splitting sets (resp.) of size less than $2n$. 

Since $\W_{{N-1}_{blue}}=\varGamma_{2_{{N-1}_{blue}}}$ and $\W_{{N-1}_{red}}=\varUpsilon_{1_{{N-1}_{red}}}$, the binomial arising from the walk on $\W_{N-1}$  is 
\[
	m_{N}u_{N}-m_{N+2}v_{N+2}=m_{N}(u_N-v_N)+m_{N+1}(u_{N+1}-v_{N+1})+m_{N+2}(u_{N+2}-v_{N+2}).
\]

Choose $e\in H$ such that $t_e \mid m_{N+1}$, then $t_e \mid v_N$ and $t_e \mid u_{N+2}$. But since $S_N=E(v_N)$ and $R_N=E(u_{N+2})$, $e \in S_N$ and $e \in R_N$, so $S_N\cap R_N\neq \emptyset$.

{\bf 
Case 3:} $r=2N+2$. 
For $1< i <N$, let
\begin{align*}
\varGamma_{1_i}&=E(u_i) \sqcup_m E(v_i)\\
\varGamma_{2_i}&=E(m_{i+1}u_{i+1}) \sqcup_m E(m_{2N-i+3}v_{2N-i+3})\\
S_i&=E(v_i) \cap E(m_{i+1}u_{i+1})=E(\gcd(v_i,  m_{i+1}u_{i+1}))=E(v_i).
\end{align*}

For $1< i <N$, let 
\begin{align*}
\varUpsilon_{1_i}&=E(m_iu_i) \sqcup_m E(m_{2N-i+2}v_{2N-i+2})\\
\varUpsilon_{2_i}&=E(u_{2N-i+3}) \sqcup_m E(v_{2N-i+3})\\
R_i&=E(m_{2N-i+2}v_{2N-i+2}) \cap E(u_{2N-i+3})
\\&=E(\gcd(m_{2N-i+2}v_{2N-i+2},u_{2N-i+3}))=E(u_{2N-i+3}).
\end{align*}

We can follow the proof of Case 1) to see that $S_1$ and $R_1$ are splitting sets of $\W$, and $S_{i+1}$ and $R_{i+1}$ are splitting sets of $\W_i=E(m_{i+1}u_{i+1}) \sqcup_m E(m_{2N-i+2}v_{2N-i+2})$ for $i=1, \ldots, N-1$. Furthermore, by definition, they are blue and red (resp.) splitting sets of size less than $n$. 
Since $\W_{{N}_{blue}}=\varGamma_{2_{{N}_{blue}}}$ and $\W_{{N}_{red}}=\varUpsilon_{1_{{N}_{red}}}$, the binomial arising from $\W_{N}$  is
\[
m_{N+1}u_{N+1}-m_{N+2}v_{N+2}=m_{N+1}(u_{N+1}-v_{N+1})+m_{N+2}(u_{N+2}-v_{N+2})
\]
which is exactly case 1), which means there exists a proper splitting set of $\W_{N}$. 

\end{proof}

\begin{proof}[Proof of sufficiency ($\Leftarrow$)]
Assume every primitive monomial walk $\W$ of length $2n >2d$ with $\supp(\W) \subset H$ satisfies $i)$ or $ii)$.  Let $p_{\W}=u-v$ be a generator of $I_H$ which arises from the monomial walk ${\W}$ on $H$.  

To show that  $I_H=[I_H]_{\leq d}$, we proceed by induction on the degree of $p_{\W}$. If $\deg p_{\W}=2$, then $p_{\W} \in [I_H]_{\leq d}$. So assume $\deg p_{\W}=n>d$ and every generator of $I_H$ of degree less than $n$ is in $[I_H]_{\leq d}$. Since the size of $\W$ is greater than $2d$, either condition $i)$ holds or condition $ii)$ holds. 

Suppose $i)$ holds. By Lemma 3.5, there exists a decomposition of $\W$, $(\varGamma_1, S, \varGamma_2)$, such that $|\varGamma_1|<|\W|$ and $|\varGamma_2|<|\W|$.  Let $p_{\varGamma_1}=u_1-v_1$ ($p_{\varGamma_2}=u_2-v_2$, respectively) be the binomial that arises from $\varGamma_1$ ($\varGamma_2$, respectively).  Let $m_1=u/u_1$ and $m_2=v/v_2$.

What remains to be shown is that $p_{\W}=m_1p_{\varGamma_1}+ m_2 p_{\varGamma_2}$, that is, $u-v=m_1(u_1-v_1)+m_2(u_2-v_2)$. However,  it is clear that $u=m_1u_1$ and $v=m_2v_2$, so it suffices to show is that $m_1v_1=m_2u_2$, or equivalently, $E(m_1v_1)=E(m_2u_2)$.

Let $s \in k[t_{e_i}]$ be the monomial such that $E(s)=S$. Then
\[
	\varGamma_1 \sqcup \varGamma_2=(E(u_1) \sqcup E(\frac{v_1}{s}) \sqcup S) \sqcup (E(\frac{u_2}{s}) \sqcup S \sqcup E(v_2))
\]
and
\[
	\W+S=(E(m_1)\sqcup E(u_1) \sqcup S) \sqcup (E(m_2) \sqcup E(v_2) \sqcup S).
\]
Thus, since $\W+S=\varGamma_1 \sqcup \varGamma_2$,
\[
	E(m_1) \sqcup E(m_2) =  E(\frac{v_1}{s}) \sqcup E(\frac{u_2}{s}),
\]
which in turn implies 
\[
	m_1m_2=(\frac{v_1}{s})(\frac{u_2}{s}).
\]
Since $\W$ is primitive and the coloring conditions on $(\varGamma_1,S,\varGamma_2)$ imply $E(\frac{v_1}{s}) \subseteq \W_{red}$ and $E(m_1) \subseteq \W_{blue}$, the monomials $m_1$ and $\frac{v_1}{s}$ are relatively prime.  A similar argument shows $m_2$ and $\frac{u_2}{s}$ are relatively prime.  Thus, $m_1=\frac{u_2}{s}$ and $m_2=\frac{v_1}{s}$, and consequently, $E(m_1v_1)=E(m_2u_2)$ and $p_w=m_1p_{\varGamma_1}+ m_2 p_{\varGamma_2}$.

Since $\deg p_{\varGamma_1}, \deg p_{\varGamma_2}<n$,  the induction hypothesis applied to $p_{\varGamma_1}$ and $p_{\varGamma_2}$ shows that $p_{\W} \in [I_H]_{\leq d}$.

Now suppose  $ii)$ holds. For $i$ from 1 to $N$, let $p_{\varGamma_{1_i}}=u_i-v_i$ and $p_{\varUpsilon_{2_i}}=y_{i}-z_{i}$ be the binomials arising from $\varGamma_{1_i}$ and $\varUpsilon_{2_i}$.  Let $w_{i_b}-w_{i_r}$ be the binomial arising from the walk $\W_i$ and let $p_{\W}=w_{0_b}-w_{0_r}$. For $1 \leq i \leq N$, let $m_i=w_{(i-1)_b}/u_i$, and $q_i=w_{(i-1)_r}/z_i$. Then
\[
	p_{\W}= \sum_{i=1}^N m_i (u_i -v_i) + w_{N_b}-w_{N_r}+ \sum_{i=1}^{N} q_{N+1-i}(y_{N+1-i}-z_{N+1-i}).
\]
The preceding claim follows from three observations: (1) by construction, $w_{0_b}=m_1u_1$ and $w_{0_r}=q_1z_1$; (2) by the definition of $\W_N$, $w_{N_b}=m_Nv_N$ and $w_{N_r}=q_{N}y_N$;  and  (3) by the definitions of $m_i$, $q_i$, and the walk $\W_i$, $m_iv_i=m_{i+1}u_{i+1}$ and $q_{i+1}z_{i+1}=q_{i}y_{i}$ for $1 \leq i \leq N-1$.  As a consequence of the size conditions on the splitting sets of $\W_i$,  the linear combination $\sum_{i=1}^N m_i (u_i -v_i) \in [I_H]_{\leq d}$ and $\sum_{i=1}^{N} q_{N+1-i}(y_{N+1-i}-z_{N+1-i})\in [I_H]_{\leq d}$. So if $\W_N$ satisfies condition $i)$, the binomial $w_{N_b}-w_{N_r}\in [I_H]_{\leq d}$, and thus, $p_{\W} \in [I_H]_{\leq d}$.

To finish the proof, assume that $S_N$ and $R_N$ share an edge, $e$. Then the claim above becomes:
\[
	p_{\W}= \sum_{i=1}^N m_i (u_i -v_i) + t_e(\frac{m_Nv_N}{t_e} - \frac{q_{N}y_{N}}{t_e})+  \sum_{i=1}^{N} q_{N+1-i}(y_{N+1-i}-z_{N+1-i})
\]
and we just need to show that, in fact, $t_e$ divides $m_Nv_N$ and $q_{N}y_{N}$. But this is clear to see since $e \in S_N$ which implies $t_e | v_N$ and $e \in R_N$ which implies $t_e | y_{N}$.

\end{proof}

\begin{eg}[{\bf Independence models}] Let $H$ be the complete $k$-partite hypergraph with $d$ vertices in each partition $V_1, \ldots, V_k$.  These hypergraphs correspond to the independence model in statistics. Equivalently, the edge subring of the complete $k$-partite hypergraph with $d$ vertices in each partition parametrizes the Segre embedding of $\mathbb P^d\times\dots\times\mathbb P^d$ with $k$ copies.

The ideal $I_H$ is generated by quadrics. To see this, let $\W$, $\supp(\W) \subseteq H$, be a primitive monomial walk of length $2n$, $n>2$. 
Choose a multiset $E'\subset \W$ consisting of $n-1$ blue and $n-1$ red edges. 
 Since each edge must contain a vertex from each $V_i$, for each $i$, there is at most one vertex in $V(E') \cap V_i$ that is not covered by a red edge and a blue edge from $E'$. Consequently,  $V(E')$  contains  a vertex from each $V_i$ that belong to at least one red edge and at least one blue edge of $E'$.

For a multiset of edges, $M$, with $\supp(M) \subseteq H$, we define the \emph{max degree} of a vertex:
\[
	\maxdeg(v; M):=\max(\deg_{red}(v; M), \deg_{blue}(v;M)).
\]
The partitioning of the vertices ensures that $V(E')$ cannot contain more then $k$ vertices whose $\maxdeg$ with respect to $E'$ is $n-1$.  Indeed, if there are more that $k$ vertices with $\maxdeg$ equal to $n-1$, then two of those vertices must belong to the same partition, $V_j$.  This would imply that $\W$ contains at least $4(n-1)$ edges, which is impossible when $n > 2$.

Next, choose $n-1$ new blue edges and $n-1$ red edges in the following manner:

Let $d_b(v):=\deg_{blue}(v; E')$ and $d_r(v):=\deg_{red}(v;E')$. For $i=1, \ldots, k$ choose a vertex from $V(E'_{blue}) \cap V(E'_{red}) \cap V_i$ that has the largest $\maxdeg$ with respect to $E'$; let $b_{n-1}$ and $r_{n-1}$ be this set of vertices. For all $v \in b_{n-1}$, reduce $d_b(v)$ and $d_r(v)$ by 1.  Now choose $b_1, \ldots, b_{n-2}$ by the following algorithm:

{\small
{\ttfamily
\quad for $i$ from 1  to k do:\\
\phantom{xxxx} let $V_i:=$sort $V(E')\cap V_i$ by $d_b(v)$ in decreasing order; \\
\phantom{x} for $j$ from $n-2$ down to 1 do: \\
\phantom{x} $($\\
\phantom{xxxx} $b_j:=$ list $\{v_i: v_i$ is first element in $V_i$\};\\
\phantom{xxxx} for all $v \in b_j$ do $d_b(v)=d_b(v)-1$;\\
\phantom{xxxx} for $i$ from 1 to $k$ do $V_i=$sort $V_i$ by $d_b(v)$ in decreasing order;\\
\phantom{x} $)$.
}
}

Let $R_1=\{b_1, \ldots, b_{n-1}\}$ and $S_1=\{r_1, \ldots, r_{n-1} \}$. Then $R_1$ and $S_1$ are red and blue splitting sets of $\W$ that share an edge.  Thus, condition ii) of Theorem~\ref{thm:uniformDegreeBound}  is met, and consequently  $I_H$ is generated in degree $2$. 
\end{eg}


When $H$ is a non-uniform hypergraph, the toric ideal $I_H$ is not necessarily homogeneous. For example, Figure~\ref{fig:nonuniform} supports a  binomial in $I_H$ where $H$ consists of edges of size two and four; note that the edges still satisfy the  balancing condition \eqref{degreeCondition}. However, we can still modify the conditions of Theorem~\ref{thm:uniformDegreeBound} to find degree bounds for the toric ideals of non-uniform hypergraphs.  Proposition~\ref{thm:nonuniform} gives a prescription for determining a degree bound on the generators of $I_H$ in terms of local structures of $H$.

\begin{figure}
\begin{center}  
\includegraphics[width=2in]{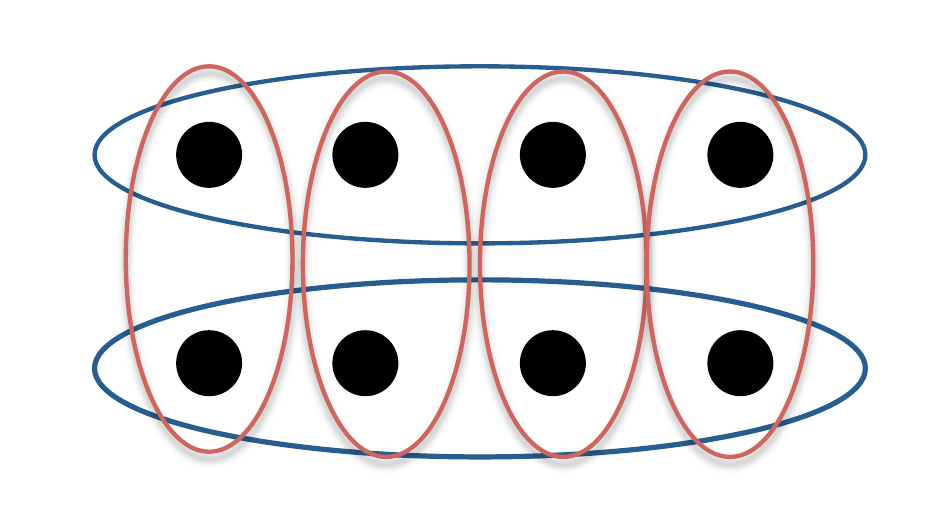}
\caption{}
\label{fig:nonuniform}
\end{center}
\end{figure}

\begin{proposition}
\label{thm:nonuniform}
	Given a hypergraph $H$ and  a binomial $f_{\E}\in I_H$ arising from the balanced edge set $\E$ with $n=|\E_{blue}| \geq |\E_{red}|$,  $f_{\E}$ is a linear combination of binomials in $I_H$ of degree less than $n$ if one of the following two conditions hold:

	i) there exists a proper splitting set $S$ of $\E$ with decomposition $(\varGamma_1, S, \varGamma_2)$ where $|\varGamma_{i_{blue}}|, |\varGamma_{i_{red}}| < n$ for $i=1,2$,

	\emph{or}
	
	ii) there is a pair of blue and red splitting sets of $\E$, $S$ and $R$, of size less than $n$ with decompositions $(\varGamma_1, S, \varGamma_2), (\varUpsilon_1, R, \varUpsilon_2)$ such that $|\varGamma_{1_{blue}}|, |\varUpsilon_{2_{red}}| < n$, $|\varGamma_{2_{blue}}|, |\varUpsilon_{1_{red}}| \leq n$, and $S \cap R\neq \emptyset$. 
\end{proposition} 

\begin{proof} This proof follows the proof of sufficiency for Theorem~\ref{thm:uniformDegreeBound}. Note that in the proof, the uniform condition doesn't play an essential role; it is only invoked to bound the size of the red and blue parts of each monomial hypergraph appearing in the  decompositions involved. Thus, the hypothesis of Proposition~\ref{thm:nonuniform} acts in  place of the uniform condition in Theorem~\ref{thm:uniformDegreeBound}.
\end{proof}

We close with an application.

For the remainder of this section,  we will concern ourselves with  the first tangential variety, $\Tan((\PP^1)^n)$. 
In \cite{SZ}, Sturmfels and Zwiernik use cumulants to give a monomial parameterization of  $\Tan((\PP^1)^n)$. The variety  $\Tan((\PP^1)^n)$ is associated to a class of hidden subset models \cite[Example 5.2]{SZ}, and context-specific independence models \cite{Oed}.
We now  derive a bound for the toric ideal of the image of $\Tan((\PP^1)^n)$ in higher cumulants and, equivalently, for the Markov complexity of these models. 

\begin{eg}\label{eg:cumulants}
Let $H=(V,E)$ where $V=\{1, \ldots, n\}$ and $E=\{ e \ : \ e \subseteq V \text{ and } |e| \geq 2\}$. Then the set of polynomials vanishing on the image of $\Tan((\PP^1)^n)$ in higher cumulants is the toric ideal $I_H$ (see  \cite[Theorem 4.1]{SZ}). 
\end{eg}

The hypergraph in Example~\ref{eg:cumulants} is the complete hypergraph on $n$ vertices after removing all singleton edges. The degree bound on the generators of this hypergraph can be found by looking at a smaller hypergraph.

\begin{lemma}\label{prop:bounds} Let $H_1=(V,E_1)$ where $V=\{1, \ldots, n\}$ and  $E_1=\{ e \ : \ e \subseteq V \text{ and } |e| \geq 2\}$, and let $H_2=(V, E_2)$ where  $E_2=\{ e \subseteq V \ : 2 \leq |e| \leq 3 \}$. If the ideal $I_{H_2}$ is generated in  degree at most $d$, then the ideal $I_{H_1}$ is generated in  degree at most $d$.
\end{lemma}

\begin{proof} Consider $I_{H_2}$ as an ideal in the bigger polynomial ring $S:=k[t_{e_i} \ : \ e_i \in H_1]$, denoted as $\tilde I_{H_2}:=I_{H_2}S$.
Assume that $I_{H_2}$, and consequently, $\tilde I_{H_2}$, is generated in degree at most $d$. Pick an arbitrary binomial 
\[u-v=t_{e_{i_1}}t_{e_{i_2}}\cdots t_{e_{i_n}}-t_{e_{j_1}}t_{e_{j_2}}\cdots t_{e_{j_m}} \in I_{H_1}.\]
Since every edge $e \in H_1$ is the disjoint union of a collection of edges $e_{k_1}, \ldots, e_{k_l} \in H_2$, we may write $t_{e}-\prod_{i=1}^l t_{e_{k_i}}  \in I_{H_1}$.  
Noting that 
\[ 
	t_{e}-\prod_{i=1}^l t_{e_{k_i}}=(t_{e} - t_{e_{k_1}}t_{\cup_{i=2}^l e_{k_i}}) - \sum_{j=1}^{l-2} \left[\left(\prod_{i=1}^j t_{e_{k_i}}\right)(t_{\cup_{i=j+1}^l e_{k_i}} - t_{e_{j+1}}t_{\cup_{i=j+2}^l e_{k_i}})\right],
\]
one easily sees that the binomial $t_{e}-\prod_{i=1}^l t_{e_{k_i}}$ is generated by quadratics.  
In turn, this essentially shows that relations in $I_{H_2}$ allow us to rewrite $u-v$ in terms of edges $e_{i_1}, \ldots, e_{i_n}, e_{j_1}, \ldots, e_{j_m} \in E_2$ of size $2$ and $3$ only. 
The claim follows since $u-v$ can be expressed as  a binomial in $\tilde I_{H_2}$. 
\end{proof}

\begin{theorem} 
\label{thm:cumulants} 
	Let $H=(V,E)$ where $V=\{1, \ldots, n\}$ and $E=\{ e \subseteq V \ : 2 \leq |e| \leq 3 \}$.  The toric ideal of $H$ is generated by quadrics and cubics. 
	
	In particular the image of $\Tan((\PP^1)^n)$ in higher cumulants is generated in degrees $2$ and $3$. 
\end{theorem}

In the following proof, we examine the local combinatorics of $H$ to illustrate how the structure of a hypergraph reveals insights into the generating set of $I_H$.

\begin{proof} 
Let $f_{\E}$ be a primitive binomial in $I_H$ with $\E$ a balanced edge set. 
Without loss of generality, we will assume throughout the proof $|\E_{blue}| \geq |\E_{red}|$. 
If $\E$ contains only 2-edges or only 3-edges, then by \cite[Theorem 14.1]{GB} $f_{\E}$ is a linear combination of quadratics.  So we will assume $\E$ contains a 2-edge and a 3-edge.  

Since $|\E_{blue}| \geq |\E_{red}|$, $\E_{blue}$ must contain at least as many 2-edges as $\E_{red}$, and in order to satisfy (\ref{degreeCondition}), the difference between the number of 3-edges in $E_{red}$ and the number of 3-edges in $E_{blue}$ must be a multiple of 2.  

Notice that for every pair $e_1,e_2$ of $3$-edges (where $e_1$ and $e_2$ do not need to be unique), there are three 2-edges in $H$, $e_3,e_4,e_5$, such that
\[ \{e_1, e_2\} \sqcup_m \{e_3, e_4, e_5\} \]
is a balanced edge set. Let $\mathcal{B}_{2,3} \subset I_H$ be the set of all binomials arising from balanced edge sets of this form.  Then $f_{\E}$ is a linear combination of binomials in $\mathcal{B}_{2,3}$ and $f_{\E'}$, where $\E'_{blue}$ and $\E'_{red}$ contains the same number of 2-edges and exactly one 3-edge.

Since it suffices to consider primitive binomials, we will proceed inductively by showing that every primitive degree $n$ binomial in
\[\mathcal{B}_h:=\{f_{\E} \in I_H : |\E_{blue}|=|\E_{red}| \text{ and } 
\E_{blue}, \E_{red} \text{ contain exactly one 3-edge  each} \}\]
is a linear combination of binomials in $\mathcal{B}_h$ with degree less than $n$. 

Let $f_{\E} \in \mathcal{B}_h$ such that degree $f_{\E}=n>3$ and $f_{\E}$ is primitive. Let $e_1$ be the 3-edge in $\E_{red}$.  Since $f_{\E}$ is primitive, $e_1$ must intersect a 2-edge $e_2$ in $\E_{blue}$. Let $e_2=\{v_1,v_2\}$ where $v_1 \in e_1$. 

The edge $e_2$ intersects at most one other edge of $\E_{red}$ besides $e_1$. We will examine the possible intersections of $e_2$ and $\E_{red}$ in order to find splitting sets of $\E$ that satisfy one of the conditions listed in Proposition~\ref{thm:nonuniform}. For illustrations of Case 1 and Case 3 see Figures \ref{fig:case1} and \ref{fig:case3}.
In all three cases, we will construct $S$, $\Gamma_1$ and $\Gamma_2$ such that $S$ is a splitting set of $\E$ with an associated decomposition $(\varGamma_1, S, \varGamma_2)$ which satisfies the properties of condition $i)$ in Theorem~\ref{thm:nonuniform}. In fact, $f_{\E}$ will be a linear combination of $f_{\Gamma_1}$ and $f_{\Gamma_2}$, both of which have strictly lower degree than $f_{\E}$. Furthermore, since the blue and red parts of $\Gamma_1$ and $\Gamma_2$ will contain the same number of 2 and 3-edges, it follows that $f_{\Gamma_1}, f_{\Gamma_2} \in \mathcal{B}_h$. 

\emph{Case 1: The edge $e_1=e_2 \cup \{v_3\}=\{v_1, v_2, v_3\}$ for some $v_3 \in V(\E)$.}\\
Since $v_3 \notin e_2$ and $|\E_{blue}|=|\E_{red}|$, there must be a 2-edge $e_3 \in \E_{red}$ such that $v_3 \notin e_3$ in order for (\ref{degreeCondition}) to hold.  Let $e_3=\{v_4,v_5\}$ and $e_4=\{v_3, v_4,v_5\}$. The sets $S$, $\Gamma_1$ and $\Gamma_2$ in this case are:
\begin{align*}
S&=\{e_4\}\\
\varGamma_1&=(\E_{blue}-\{e_2\}) \sqcup_m ((\E_{red}-\{e_1, e_3\})\sqcup \{e_4\})\\
\varGamma_2 &=\{e_2, e_4\} \sqcup_m \{e_1, e_3\}.
\end{align*}

\begin{figure}[h]
  \hfill
  \hfill
  \begin{minipage}[t]{.45\textwidth}
    \begin{center}  
\includegraphics[width=1.5in]{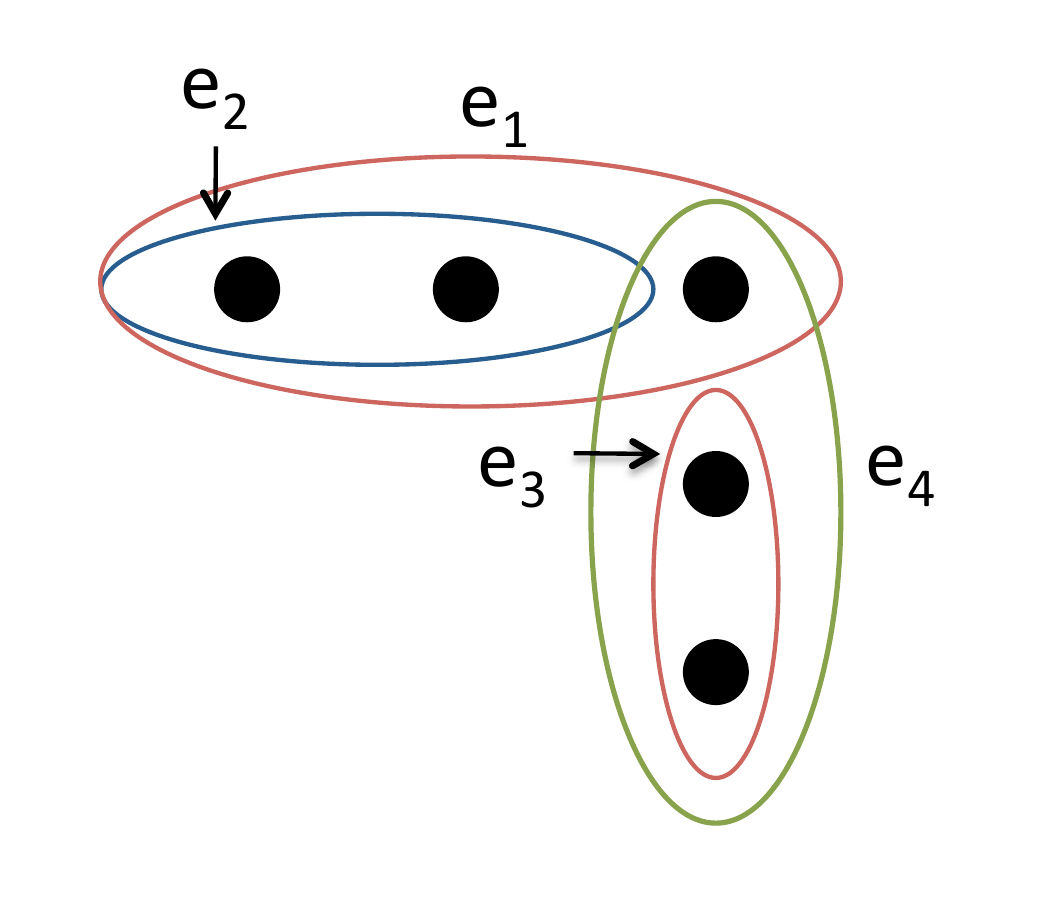}
\caption{Case 1}
\label{fig:case1}
    \end{center}
  \end{minipage}
  \hfill
  \begin{minipage}[t]{.45\textwidth}
    \begin{center}  
\includegraphics[width=1.2in]{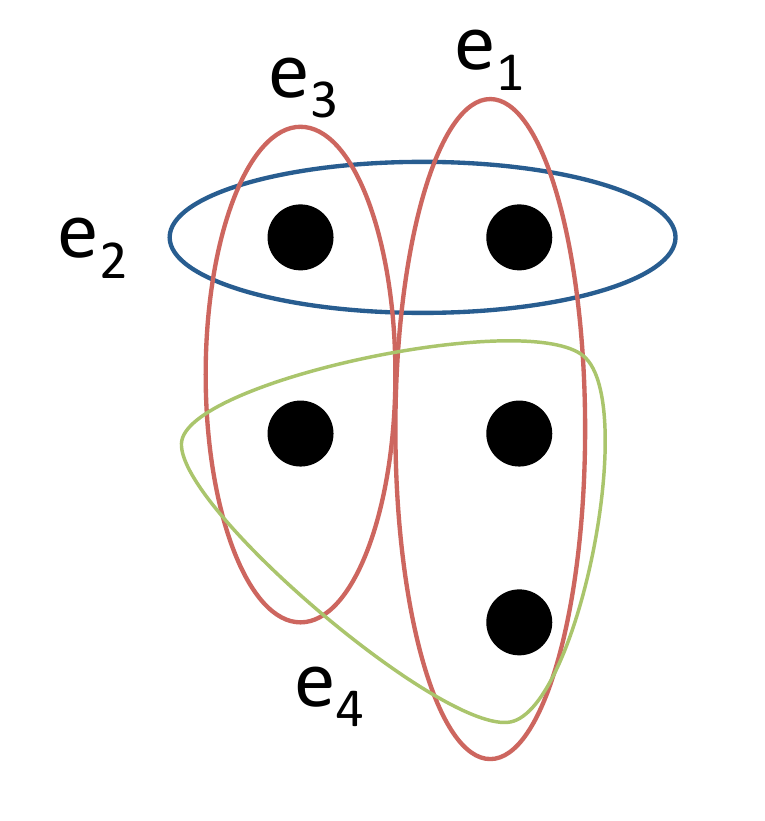}
\caption{Case 3}
\label{fig:case3}
    \end{center}
  \end{minipage}
  \hfill
  \hfill
\end{figure}

\emph{Case 2:  The edge $e_1=\{v_1,v_3,v_4\}$ for some $v_3,v_4 \in V(\E)$ and there is a 2-edge $e_3 \in \E_{red}$ such that $e_3=\{v_2,v_3\}$}.

Since $v_3 \notin e_2$, $\deg_{blue}(v_3;\E)=\deg_{red}(v_3;\E) \leq n-1$ and, thus, there exists a 2-edge $e_4 \in \E_{red}$ such that $v_3 \notin e_4$.  Let $e_4=\{v_5,v_6\}$.

Now let $e_5=\{v_3, v_4, v_5\}$ and $e_6=\{v_3,v_6\}$. The sets $S$, $\Gamma_1$ and $\Gamma_2$ in this case are:
\begin{align*}
S&=\{e_5,e_6\}\\
\varGamma_1&=(\E_{blue}-\{e_2\}) \sqcup_m ((\E_{red}-\{e_1, e_3,e_4\})\sqcup \{e_5,e_6\})\\
\varGamma_2 &=\{e_2,e_5,e_6\} \sqcup_m \{e_1, e_3, e_4\}.
\end{align*}

\emph{Case 3: There is a 2-edge $e_3 \in \E_{red}$ such that $v_2 \in e_3$ and $e_2 \cap e_3 = \emptyset$}.\
In this case, let $e_4=(e_1-\{v_1\}) \cup (e_3 - \{v_2\})$. The sets $S$, $\Gamma_1$ and $\Gamma_2$ in this case are:
\begin{align*}
S&=\{e_4\}\\
\varGamma_1&=(\E_{blue}-\{e_2\}) \sqcup_m ((\E_{red}-\{e_1, e_3\})\sqcup \{e_4\})\\
\varGamma_2 &=\{e_2,e_4\} \sqcup_m \{e_1, e_3\}.
\end{align*}
\end{proof}

\section*{Acknowledgements} The authors would like to thank the anonymous referee for carefully reading the previous version of this manuscript and thus allowing us to greatly improve Section~\ref{sec:nonuniform}. We are also grateful to Despina Stasi and Seth Sullivant for helpful discussions.



\begin{thebibliography}{10}

\bibitem{AA} A.~Takemura and S.~Aoki. {Some characterizations of minimal Markov
basis for sampling from discrete conditional distributions}, Ann. Inst. Statist. Math.
{\bf 56} 1 (2004), 1Ð17.

\bibitem{ATY} S.~Aoki, A.~ Takemura, R.~Yoshida. {Indispensable monomials of toric ideals and Markov bases}, Journal of Symbolic Computation {\bf 43} 6Ð7 (2008) 490-507.

\bibitem{CKT} H.~Charalambous, A.~Katsabekis, and A.~Thoma. {Minimal systems of binomial generators and the indispensable complex of a toric ideal}, Proceedings of the American Mathematical Society {\bf 135} (2007) 3443-3451.


\bibitem{DLO} J.~De Loera and S.~Onn,  {Markov bases of three-way tables are arbitrarily complicated}, J. Symb. Comput. {\bf 41} 2 (February 2006) 173-18.

\bibitem{DS} M.~Develin and S. Sullivant. {Markov bases of binary graph models}, Annals of Combinatorics  {\bf 7} (2003) 441-466.


\bibitem{DiacSturm}
P.~Diaconis and B.~Sturmfels. {Algebraic algorithms for sampling  from conditional distributions}, Ann. Statist. \textbf{26}, no.~1,  363--397 (1998) 

\bibitem{DobraSul} A.~Dobra and S.~Sullivant. A divide-and-conquer algorithm for generating Markov bases of multi-way tables.   Computational Statistics  19 (2004), 347-366


\bibitem{algStatBook} M.~Drton, B.~Sturmfels and S.~Sullivant. \emph{Lectures on algebraic  statistics}, Oberwolfach Seminars \textbf{39}, Birkh\"auser  (2009)

\bibitem{GRV} I.~Gitler, E.~Reyes, R.~Villarreal. {Ring graphs and toric ideals}, Electronic Notes in Discrete Mathematics {\bf 28} 1 (2007) 393-400.

\bibitem{HSull} S.~Ho\c{s}ten and S.~ Sullivant. {A finiteness theorem for Markov bases of hierarchical models}, J. Comb. Theory Ser. A {\bf 114} 2 (2007) 311-321. 

\bibitem{Oed} L.~Oeding. {Set-theoretic defining equations of the tangential variety of the
Segre variety}, J. Pure and Applied Algebra, {\bf 215} (2011) 1516-1527.
 
 \bibitem{OH}  H.~Ohsugi and T.~ Hibi. Toric ideals generated by quadratic binomials, Journal of Algebra {\bf 218} (1999), 509-527.

\bibitem{OH2005} H.~Ohsugi and T.~Hibi. {Indispensable binomials of finite graphs}, J. Algebra Appl. {\bf 4} (2005), no 4, 421-434.

\bibitem{PS} S.~Petrovi\'{c} and D. Stasi. Toric algebra of hypergraphs. Preprint: \href{http://arxiv.org/abs/1206.1904}{arXiv:1206.1904}

\bibitem{RTT} E.~Reyes, C.~Tatakis, A.~Thoma. {Minimal generators of toric ideals of graphs}, Adv. in Appl. Math {\bf 48} (2012), no. 1, 64-67

\bibitem{GB} B.~Sturmfels. \emph{Gr\"{o}bner bases and convex polytopes}, University Lecture Series, 8. American Mathematical Society, 1996.

\bibitem{phylo} B.~Sturmfels and S.~Sullivant. {Toric ideals of phylogenetic invariants},  Journal of Computational Biology {\bf 12} (2005) 204-228.

\bibitem  {SZ} B.~Stumfels and P.~Zwiernik. Binary cumulant varieties, Annals of Combinatorics, to appear.

\bibitem{TT} C.~Tatakis andA.~Thoma. {On the universal Gr\"{o}bner basis of toric ideals of graphs}, Journal of Combinatorial Theory, Series A, {\bf 118} (2011) 1540-1548

\bibitem{Vill95}
R.~Villarreal. Rees algebras of edge ideals, Communications in Algebra, {\bf 23} (9), 3513--3524 (1995)

\bibitem{Vill}
R.~Villarreal, \emph{Monomial Algebras}. Monographs and Textbooks in Pure and Applied Mathematics {\bf 238}, Marcel Dekker, Inc., New York, 2001.
\end{thebibliography}
\end{document}